\documentclass[letterpaper, 12pt]{article}
\usepackage{palatino}
\title{The heat and the landscape I}
\author{B\'alint Vir\'ag}
\date{}

\usepackage{amsmath, amsthm, amssymb}
\usepackage{enumerate}
\usepackage{color}

\usepackage[longnamesfirst]{natbib}
\bibpunct[ ]{(}{)}{,}{a}{}{,}

    \newtheorem{theorem}{Theorem}
    \newtheorem{lemma}[theorem]{Lemma}
    \newtheorem{proposition}[theorem]{Proposition}
    \newtheorem{corollary}[theorem]{Corollary}

\theoremstyle{definition} % For roman text in the body
    
    \newtheorem{remark}[theorem]{Remark}

\newcommand{\eps}{\varepsilon}
\newcommand{\Z}{{\mathbb Z}}

\newcommand{\R}{{\mathbb R}}

\newcommand{\cd}{\stackrel{d}{\longrightarrow}}
\newcommand{\cp}{\stackrel{p}{\longrightarrow}}
\newcommand{\ed}{\stackrel{d}{=}}
\newcommand{\cB}{{\mathcal B}}

\newcommand{\cK}{{\mathcal K}}
\newcommand{\cS}{{\mathcal S}}

\newcommand{\cR}{{\mathcal R}}
\newcommand{\cO}{{\mathcal O}}
\newcommand{\cf}{{\mathfrak f}}
\newcommand{\blp}{\text{\sc blp}}
\newcommand{\ep}{\text{\sc ep}}
\newcommand{\uc}{\text{\sc uc}}

\newcommand{\supn}{\operatorname{supn}}%{\operatorname{sup}_n}

\newcommand{\ncirc}{\circ_{\mbox{\tiny n}}}

\newcommand{\one}{\mathbf 1}
\newcommand{\zero}{\mathbf 0}

\begin{document}

\maketitle
\vspace{-3em}
\center{University of Toronto}

\medskip

\abstract{{\it Heat }flows in 1+1 dimensional stochastic environment converge after scaling to the random geometry described by the directed {\it landscape.} In this first part, we show that the O'Connell-Yor polymer and the KPZ equation converge to the KPZ fixed point.
%In the second part, we show convergence to the directed landscape.

The key is that one-dimensional  Baik-Ben\,Arous-Peche statistics characterize the KPZ fixed point. This yields a general and elementary method that shows convergence based on previously established limit theorems.
%We apply this method to the O'Connell-Yor polymer and the KPZ equation.

Independently, at the {\it same time and place}
\cite{QS} gave an unrelated proof of KPZ convergence. The methods invite extensions in different directions: ours to polymer models, and theirs to interacting particle systems.
}

\section{Introduction}

Some important interacting particle systems on the line, some last passage models and some polymer models share a defining geometric structure. This is believed to be universal, in the sense that after scaling, it should converge to the directed landscape, a random directed metric on the plane. Key examples are the classic KPZ equation of \cite{kardar1986dynamic} or the equivalent continuum random directed polymer model, and the O'Connell-Yor polymer. The precise description of these models and scaling for the following theorem is given in Section \ref{s:application}.

\begin{theorem}\label{t:main}
Let $H_n(t,x)$ denote the free energy of the appropriately rescaled O'Connell-Yor model or the solution of the KPZ equation started from a deterministic continuous initial condition $H(0,x)\le c+b|x|$ for all $x$. Then for any finite set of triples $(t_i,x_i,a_i)$
$$
P(H_n(t_i,x_i) \le a_i\forall i ) \to P(H(x_i,t_i)\le a_i\forall i ),
$$
where $H$ is the KPZ fixed point started from $H(0,\cdot)$. For fixed times $t>0$ we have compact convergence in law.
\end{theorem}
The theorem holds for more general initial conditions which take longer to describe, including all classically considered ones,  see Corollary \ref{c:compact}. 

The proof is based on a new characterization of the KPZ fixed point. We show that its law is determined by a simple collection of observables. It is best to formulate this in terms of a hands-on representation of the KPZ fixed point given by the Airy sheet, see \cite{DOV} for a precise definition.

The Airy sheet is a random continuous function $\mathbb \R^2\to \mathbb R$. It maps upper semicontinuous ($\uc$) functions to random functions via the metric composition
$$
f\mapsto f\cdot \cS, \qquad f\cdot g=\sup f+g, \qquad (f\cdot \cS)(y)=  f(\cdot)\cdot \cS(\cdot,y).
$$
This gives a Markov transition kernel on the space of functions, called the KPZ fixed point, introduced by \cite{matetski2016kpz}. It is an open problem whether this kernel determines the law of the Airy sheet. The KPZ fixed point is characterized by the functional
\begin{align}\nonumber
\uc^2 &\to
\mbox{probability measures on }  \R\cup\{\pm \infty\}\\
(f,g)&\mapsto \mbox{Law}(f\cdot \cS \cdot g). \label{e:fg}
\end{align}
The law can be described in terms of cumulative distribution function $a\mapsto P(f\cdot \cS \cdot g\le a)$, the probability that the fixed point started from initial condition $f$ is bounded above by $a-g$ at time 1. In fact, this is  the characterization used in the original definition in \cite{matetski2016kpz}. We call this functional the {\bf distance laws} of $\cS$.

The KPZ fixed point is equivalent to the distance laws of the Airy sheet.

The {\bf edge process} $\ep(a,h,\nu)$ is the distribution of {\bf top line} of Brownian motions started at time $a$, nondecreasing height vector $h$ with nondecreasing drift vector $\nu$, and conditioned not to intersect.
We call an edge process {\bf pointed} if the heights are all the same. Pointed edge processes are the distribution of the top eigenvalue of drifted Hermitian Brownian motion $\operatorname{GUE}_{t-a}+hI+\operatorname{diag}(\nu)t$. They are also the distributions of Brownian last passage with drift $\nu$, even when the entries of $\nu$ are reordered.

\cite{baik2005phase} show that in a certain scaling, the law of the top eigenvalue at a fixed time, as the dimension $n\to\infty$ with all but finitely many $\nu$ equal to zero, converges to what are now called Baik-Ben\,Arous-Peche (BBP) laws, a modification of the Tracy-Widom law. These laws are also followed by top eigenvalues of the stochastic Airy operators under finite rank perturbations, see \cite{bloemendal2013limits} and \cite{bloemendal2016limits}.

In Brownian last passage, the BBP results can be interpreted as follows. As Browninan last passage converges to the Airy sheet, we can keep a few drifted lines on the top and on the bottom. In the right scaling, these become random initial conditions and random ``test functions'' $(f,g)$ as in $\eqref{e:fg}$. The fix point distributions converge to the BBP laws. As long as the laws of these pairs $(f,g)$ are dense in the appropriate sense, they will characterize the distance laws of $\cS$.

The following theorem can be used to characterize the KPZ fixed point. The proof uses the work of \cite{dieker2008determinantal}, which can be used to represent the law of the top line discrete last passage percolation as linear combination of top line laws for nointersecting random walks.

\begin{theorem}[Characterization theorem] \label{t:char} Let $\cS_1,\cS_2: \mathbb R^2 \mapsto \mathbb R\cup \{-\infty\}$ be random upper semicontinuous functions so that $\sup_u \cS_i(u)/(1+\|u\|)$ are finite.

Assume that for all pointed edge processes $F,G(\cdot -)$ with $F,G,S_i$ independent, $F\cdot \cS_1\cdot G \ed F\cdot \cS_2\cdot G$. Then $\cS_1,\cS_2$ have the same distance laws.
\end{theorem}

We will use this characterization to prove the following. Here $\ncirc$ denotes the geometric analogue of metric composition $\cdot$ that changes over $n$ because of scaling.  See Section \ref{s:first} for details. Given a function $f:\R\to \R$, its line metric ${\mathfrak f}\in \uc(\R^2)$ is defined as
\begin{equation}\label{e:line metric}
{\mathfrak f}(x,y)=\begin{cases}f(y)-f(x)\qquad &x\le y,\\
-\infty &y<x.
\end{cases}
\end{equation}
In our two applications of the next theorem, the functions $B_{n,\nu}=B_\nu$ will be Brownian motions with drift $\nu$ and variance 2.

\begin{theorem}[Convergence theorem]\label{t:second}
Assume that we have a sequence of random upper semicontinuous functions $\cK_n:\mathbb R^2\to \mathbb R\cup\{-\infty\}$ and and a family of random functions $B_{\nu,n}$, with increments stochastically increasing in $\nu\in \R$, with line metrics \eqref{e:line metric} $\cB_{\nu,n}$ so that for every $\nu<\mu$ we have
\begin{enumerate}[(i)]
    \item Metric Burke property: with $\cR=\cK_n$ or $\cB_{\mu,n}$, we have $\cB_{\nu,n} \ncirc \cR\ed \cR \ncirc \cB_{\nu,n}$.
    \item Tightness in stationarity: with $\cR=\cK_n$ or $\cB_{\mu,n}$, we have $B_{\nu,n}\ncirc \cR \ed B_{\nu,n}+C$ with $C^+$ tight in $n$.
    \item Near-linearity: for every $\eps>0$, $\sup |B_{\nu,n}(x)-\nu x|-\eps|x|$  is tight in $n$.
    \item Brownian limit: $B_{\nu,n}(2x)-2\nu x$ converges compactly in law to standard two-sided Brownian motion.
    \item BBP asymptotics: for every $x,y\in \R$, and every vector $\nu_1,\ldots, \nu_k$ with a sufficiently small maximum
        $$
        \cB_{\nu_1,n} \ncirc \cdots \ncirc \cB_{\nu_k,n}\ncirc \cK_n(x,y) \cd \cB_{\nu_1} \cdots \cB_{\nu_k}\cdot \cS (x,y).
        $$
\end{enumerate}
Let $\cS$ be the Airy sheet. For every continuous initial condition $h$ with $\sup h(x)/(1+|x|)$ finite, we have
$$h\,\ncirc\, \cK_n \cd h\,\cdot \,\cS  \quad \mbox{ compactly. }$$
\end{theorem}

Theorem \ref{t:main} follows from Theorem \ref{t:second} after verifying the conditions, Section 6. The only tightness or limit theorems we rely on are the BBP convergence theorem of \cite{borodin2014free}, and the stationary tightness bounds  \cite{seppalainen2010bounds} and \cite{balazs2011fluctuation}.

\medskip

%$K$ is skew-symmetric if $K(x,y)=K(-y,-x)$.
{\bf \noindent Consequences.}
Note that the precise statement of Theorem 1 makes sense but is false for narrow wedges; $\zero_0$ is indistinguishable from  $0$ by polymer models. But the right version of Theorem 1 holds for narrow wedges, see Section 5. Thus the top line of the Whittaker process or the KPZ line ensemble, see \cite{corwin2016kpz}, converges compactly in law to the top line of the Airy process. \cite{corwin2016kpz} have  strong tightness results for KPZ; we do not use them here.

The top Airy line limit of the KPZ equation and tightness of the KPZ line ensemble implies convergence to the Airy line ensemble, as shown in \cite{dimitrov2020characterization}.

The arguments in this paper are set up to work for discrete polymers as well -- essentially only the notion of $\supn$ needs to be changed to reflect discrete polymer composition. In fact, we allow changing test functions, which are not used for the two models we consider! The only missing part for log-gamma polymer models is that BBP marginals have not been shown. It seem that the technology is there. There are BBP marginal results for ASEP and the six-vertex models by \cite{aggarwal2019phase}, that are long to state, but the right interpretation should make the present analysis usable for those models as well.

Our method directly applies to all discrete last passage models where BBP convergence has been shown.

In the part II, we will use the methods introduced \cite{DOV} to show convergence to the directed landscape. It has to be modified, but for key parts useful KPZ versions have been worked out. The Airy sheet level information is contained in the KPZ line ensemble, which converges by our results and \cite{dimitrov2020characterization}. Such results go back to \cite{noumi2002tropical}, see also \cite{biane2005littelmann} and of \cite{corwin2020invariance}.

\medskip

{\bf \noindent Some notation.} $\uc$ is the space of upper semicountinous functions $\R\to\R\cup\{-\infty\}$, compactified at $-\infty$, with respect to local hypograph topology. It can be represented by mapping the hypograph of a function (including the values $-\infty$) through $(x,y)\mapsto (\arctan(x),\arctan(y)/(1+x^2))$ and using the Hausdorff metric on subsets of $(\pm\pi/2)\times \mathbb R$. Functions dominated by a fixed function are precompact in $\uc$. The space $\uc(\R^2)$ of functions from $\R^2$ is constructed similarly. A commonly used function in $\uc$ is $\zero_A$. It equals $0$ on the closed set $A$, and $-\infty$ elsewhere. The function $\zero_{a}=\zero_{\{a\}}$ is called the {\bf narrow wedge} at $a$. Brownian motions are assumed to have variance $2$ to comply with the Airy sheet convention. Compact convergence means uniform convergence on compact sets; compact convergence in law refers to weak convergence with respect to this topology.

\section{Last passage and edge processes}

\subsection*{Brownian last passage measures}

Let $a\in \mathbb R$, let $h=(h_1\le \dots \le h_n)$ be reals. Let $\blp_{a,h,\nu}$ denote the distribution of the Brownian last passage with drift vector $\nu$ on $[a, \infty)$ with initial condition vector $h$. That is, $\blp_{a,h,\nu}$ is the distribution of the random function
\begin{equation}\label{e:blp}
L(y) = \max_{1\le \ell \le n} h_\ell+\;\;\max_{t_{\ell-1}=a\le t_\ell\le \dots \le t_n}\;\;\sum_{i=\ell}^n B_i(t_i)-B_i(t_{i-1})
\end{equation}
when $y\ge a$ and $L(y)=-\infty$ when $y<a$, where the $B_i$ are independent copies of Brownian motion on $[a, \infty)$ with drift $\nu$.

\subsection*{Edge processes}

Some of the $\blp$ laws also arise as top lines of nonintersecting Brownian motions conditioned not to intersect.

Let $\ep_{a,h,\nu}$ denote the law of the {\bf top line} of Brownian motions, starting at time $a$ at the vector $h$ with drift vector $\nu$,  conditioned not to intersect forever. When $h, \nu$ are both strictly increasing, this event has positive probability. The law of the weakly increasing case can be obtained as a limit. We call these laws {\bf edge processes}. When all drifts are the same, we call the edge process {\bf parallel}. When all starting points are the same, we call the edge process {\bf pointed}.

It is well known that for pointed edge processes, $\ep_{a,h,\nu}=\blp_{a,h,\nu}$. When $h$ is not constant, this equality fails.  The relation of the two laws is most explicitly studied by \cite{dieker2008determinantal} in the discrete setting. We will take limits of a probabilistic consequence of their result to get
\begin{theorem}\label{t:mix}
Let $U_{i,j}$ be independent uniform$[0,1]$ random variables. Let $R_i=U_{i,1}+\ldots + U_{i,n-i}$. Let $\nu\in \R$, $h\in \R^n$, $r>0$ so that $h_i+rR_i$ is increasing in $i$ a.s. Let $S_i$ be independent Binomial$(n-i,1/2)$. Then
$$
E\,\blp(a,h+rR,\nu)=c\,E[(-1)^{S}\Delta(h/r +S)\,\,\ep(a,h+rS,\nu)], \qquad c=\frac{2^{\binom{n}{2}}}{1!2!\cdots (n-1)!}.
$$
\end{theorem}

Here $\Delta$ is the Vandermonde determinant. Informally, the top line of Brownian last passage started from the random left initial condition $h+rR$  is a linear combination of parallel edge processes with a shared drift and starting points $h+rs$ as $s$ ranges over the finite support of the law of $S$.

Theorem \ref{t:mix} follows directly from a discrete version, a consequence of the work \cite{dieker2008determinantal}. Next, we will describe the part of their results relevant to us.

\subsection*{The discrete setting}

Following \cite{dieker2008determinantal}, Case B,  consider a family $\xi(k,t); k \in \{1, 2,...,n\},t\in \mathbb  N$ of independent Bernoulli$(p)$ random
variables.  Define the process $Y$ taking values in $$\mathbb Z^{\downarrow n}=\{(y_1,\ldots, y_n)\in \mathbb Z^n, y_1\ge \cdots \ge y_n\}$$
via the recursions $Y_1(t) = Y_1(t-1) + \xi(1,t)$
and
for $k = 2, 3,...,n,$
$$
Y_k(t) =\min Y_k(t- 1) + \xi(k,t) , Y_{k-1}(t).
$$
This gives a local description of a sign-changed last passage percolation.
Start $Y$ with initial conditions $-h$ at some integer time $a$.  Then for integers $t\ge a$ we have  $Y_{n}(t)=-L(t)$ as defined by the formula \eqref{e:blp} with
\begin{equation}\label{e:DWlp}
B_k(t)=\begin{cases} -\sum_{i=1}^{\lfloor t \rfloor}\xi(k,i) , &t\in \mathbb Z,
\\\mbox{ linear} & \mbox{on }  [i,i+1], i \in  \mathbb Z.
\end{cases}
\end{equation}
Dieker and Warren give an explicit description of the transition probabilities $Q^t$ of the process $Y$. They are given in terms of the intertwining relation with the process $Z$, a vector of $n$ Bernoulli$(p)$ random walks conditioned to stay in $\mathbb Z^{\downarrow  n}$  forever, the latter being a Markov chain with some transition probabilities $P^t$. Based on \cite{dieker2008determinantal}, we show

\begin{theorem}\label{t:DW}
Let $U_{i,j}$ be an array of independent random variables uniform on  $\{1,\ldots, r\}$. Let $R_i=U_{i,1}+\ldots + U_{i,n-i}$, let ${\bf y}\in \mathbb R^n$ so that the random initial condition $Y(0)={\bf y} + R$ has increasing coordinates a.s.

Let $S_i$ be independent Binomial$(n-i,1/2)$ and set the random initial condition  $Z(0)={\bf y}+r S$. Then for all measurable $A$ the processes $Z_n$ and $Y_n$ satisfy
\begin{equation}\label{e:DW}
P(Y_n\in A) = c\,E[(-1)^{S}
\Delta({\bf y}/r+S);\, Z_n\in A],
\qquad c=\frac{2^{\binom{n}{2}}}{1!2!\cdots (n-1)!}.
\end{equation}
\end{theorem}
Again, the law of $Y_n$ is a finite linear combination of laws of $Z_n$ started deterministically.

\medskip
The proof of Theorem \ref{t:mix} based on Theorem \ref{t:DW} is a standard. The last passage functional is continuous in its parameters, the initial conditions $h_i$ and functions $B_i$ with respect to compact convergence. By scaling and subtracting the right drift, these converge in law to their continuum limit, trivially for the initial conditions and by Donsker's theorem for the paths $B$. So the last passage curve converges by the continuity theorem. Also, nonintersecting Bernoulli random walks converge to nonintersecting Brownian motions. Note the sign change as in \eqref{e:DWlp}.

\begin{proof}[Proof of Theorem \ref{t:DW}]
The intertwining relation is $PK=KQ$. Here  $K$ is an infinite transition probability matrix. It encodes couplings of $Y$ and $Z$ in the usual way: the conditional mass function of $Y$ given $Z=z$ is $K_{z,\cdot}$. Under these couplings one (and only one) entry is preserved: $Z_n=Y_n$, a standard property of RSK. This can be expressed with matrices as follows. With  $(H_\ell)_{w,w'}=\one_{w=w',w_n=\ell}$,
$$
(H_\ell K)_{z,y}=P(Z=z,Z_n=\ell,Y=y)=P(Z=z,Y_n=\ell,Y=y)=(K H_\ell)_{z,y}
$$
that is, $K$ and $H_\ell$ commute for all $\ell$. The commutation relations and $K\bar 1=\bar 1$ for the all-1 vector yield
$$
KQH_{\ell_1}QH_{\ell_2}\cdots QH_{\ell_t}\bar 1=
PH_{\ell_1}PH_{\ell_2}\cdots PH_{\ell_t}\bar 1.
$$
We use this in the formula that a Markov chain is in given sets at times $1,\ldots ,t$:
\begin{align}\nonumber
P_{y_0}(Y_n(i)=\ell_i,\,i=1,\ldots,t)&=(QH_{\ell_1}QH_{\ell_2}\cdots PH_{\ell_t}\bar 1)_{y_0},
\\ \nonumber
P_{z_0}(Z_n(i)=\ell_i,\,i=1,\ldots,t)&=(PH_{\ell_1}PH_{\ell_2}\cdots QH_{\ell_t}\bar 1)_{z_0}.
\end{align}
\cite{dieker2008determinantal} give an explicit left/right inverse $M$ of $K$, for which we showed
$$
P_y(Y_n(\cdot)\in A) = \sum_{z}M_{y,z}P_z(Z_n(\cdot)\in A).
$$
$M$ typically has negative entries, so this is not a mixture, just a linear combination. Let  $\hat z=(z_1-1,\ldots, z_n-n)$. Let
$$s_z=\prod_{1\le i<j\le n}
 \frac{z_i-z_j + j-i}{j-i}
 =c'\,\Delta(\hat z)
 \,(-1)^{\binom{n}{2}}, \qquad c'=\frac{1}{1!2!\cdots(n-1)!}
$$
be the Schur polynomial evaluated at (1,\ldots, 1), where $\Delta$ is the Vandermonde determinant. Let $p^z=p^{z_1+\ldots+z_n}$. In equation (5) \cite{dieker2008determinantal} introduce the matrix $\Lambda_{z,y}=s_zp^zK_{z,y}p^{-y}$, see also the formula for $K$ on p 1169.
They show that $\Lambda$ is invertible with left/right inverse $\Pi$. In the first and second displayed formulas on p 1171 they show that for any function $f(z):\Z^{\downarrow n}\to \mathbb R$ we have
$$
\sum_{y\in \Z^{\downarrow n}}f({\hat y})\Pi_{yz}=\sum_{\ell_1,\ldots, \ell_n\in \mathbb Z}(-1)^{\ell_1+\cdots +\ell_n}\binom{n-1}{\ell_1}\cdots \binom{1}{\ell_{n-1}}p^\ell f(\hat z+\ell)
$$
Now $\Pi_{yz}=p^yM_{yz}p^{-z}/s_z$. When applying this to the function $h(y)=f(\hat y)p^y$, the $p$ terms cancel and
$$
hM(z)=\sum_{\ell_1,\ldots, \ell_n\in \mathbb Z}(-1)^{\ell_1+\cdots +\ell_n}\binom{n-1}{\ell_1}\cdots \binom{1}{\ell_{n-1}}h(z+\ell) s_z,
$$
where $h$ is applied to vectors $y$ so that $\hat y\in \Z^{\downarrow n}$.
For any  $g:\mathbb Z\to \mathbb R$  the binomial theorem gives
$$
\sum_\ell \binom{n}{\ell}(-1)^{\ell} g(j+\ell)=\sum_{i_1,\dots, i_n=0}^{1}(-1)^{i_1+\ldots+i_n}g(j+i_1+\cdots+ i_n)=(-D)^n g(j)
$$
where $D g(j)=g(j+1)-g(j)$, so in our case
\begin{equation}\label{e:Dh}
hM(z)=c' \,D_1^{n-1} \cdots D_{n-1}^1\,h(z)\Delta(\hat z).
\end{equation}
Its continuum limit is the same with each $D_i$ replaced by $2\partial_i$, but we will not use this.

To check \eqref{e:DW}, take $h$ to be the mass function of $S$ in \eqref{e:Dh}. Using convolution powers of the mass function  $u=\one_{\{1,\ldots,r\}}/r$ of $U_{ij}$ we have
$$
h(z)=\prod_{i=1}^n u^{*(n-i)}(z_i), \qquad  D_1^{n-1} \cdots D_{n-1}^1\,h(z) =\prod_{i=1}^n (Du)^{*(n-i)}(z_i),
$$
since $D$ commutes with convolution.  $Du=(\one_0-\one_r)/r$ which can be compared to the Bernoulli$(1/2)$ distribution on $\{0,r\}$.
Equation \eqref{e:DW} now  follows from \eqref{e:Dh}, after we hide the factor $r^{-\binom{n}{2}}$ in $\Delta$.
\end{proof}
\section{The characterization theorem}

The characterization theorem is based on the fact that finite linear combinations of  pointed edge process laws are dense in the space of probability measures on $\uc$. A more  precise variant of this is the following.

\begin{theorem} \label{t:density}
Let  $f\in\uc$ be Lipschitz continuous on a compact interval $A$ and $-\infty$ outside. There exists a random sequence of functions $F_n$ in $\uc$ so that as $n\to\infty$, a.s.
\begin{enumerate}[(i)]
\item $F_n \to f$ in  $\uc$ from above,   and compactly on $A$
\item $\limsup\sup_x F_n(x) +b|x|<c_b,$  a deterministic constant for every $b\ge 0$,
\item the law of $F_n$ is a finite linear combination
of pointed edge processes started at the same point $-a_n$.
\end{enumerate}
\end{theorem}

The proof of this theorem consists of four steps. First, we approximate $f$ using Brownian last passage percolation started from some left initial condition at time $-a_n$. We then randomize these initial conditions slightly to set up for the next step. Next, we represent the approximation as a finite linear combination of parallel edge processes. Finally, we use time inversion to transform them to pointed edge processes.

Let $W_{n}$ be the top line of $n$ nonintersecting Brownian motions. We will use that
$$
W_{n}(a+t)-(2a+t)\sqrt{2n/a}$$
is close to the narrow wedge $\zero_0$. We first show this in a standardized setting.
\begin{proposition}\label{p:wedge} Let $W_n'$ be the top curve of {\bf standard }nonintersecting Brownian motions,  let $\alpha\in(-1/6,1)$, and let  $\Lambda_n'(t) = W'_n(t)-\sqrt{n}(t+1)$. Then for some $c,c',d,\delta>0$ and all large $n$ with probability at least $1-ce^{-dn^\delta}$,
for all $0\le h\le c'n^{1/2}$
we have
$$
-\sqrt{n}(t-1)^2- n^{\alpha}\le \Lambda'_n(t), \quad t\in [1/2, 3/2], \qquad
\Lambda'_n(t)\le n^{\alpha}+ h^2-\tfrac{2}{3}|t-1|h n^{1/4}, \quad t\ge 0.
$$
%
%Let $1\le a^3 \le n$, $\eps\in (0,1)$, $\rho=\eps(\frac{n}{2a^3})^{1/4}$, $y=-4am/n\in [-b,0]$.
%Then with probability at least $1-\exp(-\cdot)$, we have
%$$
%-\eps^2 < \Lambda_{a,n,m}(t)\quad  \mbox{ for all } t\in [y\pm \eps/2], \qquad   \Lambda_{a,n,m}(t)<2\eps^2-\rho |t-y|\quad \mbox{ for all } t\ge -a.
%$$
\end{proposition}
\begin{proof} Proposition 4.3 of \cite{DV} implies that with the above high probability,
$$
W_n'-\sqrt{n}(t+1)\le w_n(t), \quad w_n(t) =2\sqrt{nt} -\sqrt{n}(t+1)+ n^{\alpha}(1+t^{2/3}), \quad   t\ge 0 $$
The curve $w_n$ is concave, and $\partial_t w_n(2)\ge -c'\sqrt{n}$. By monotonicity, all tangent lines to $w_n$ with slopes at least than  $-c'\sqrt{n}$ will have to hit the curve in $[0,2]$. But on $[0,2]$, $w_n(t)\le v_n(t)=-\sqrt{n}(t-1)^2/9+ n^{\alpha}$. So all tangent lines to $v_n$ with slopes at least $-c'\sqrt{n}$ are upper bounds for $\Lambda_n'$. The second claim follows after computing these lines.

In the first claim, the value of the bounding curve is at least $n^{\alpha}$ below the typical location $2\sqrt{nt}-\sqrt{n}(t+1)$. A single point bound using \cite{ledoux2010small}, (see also Theorem 3.1 in \cite{DV}), applied
at a polynomial number of periodic points, and the modulus-of-continuity bound, Corollary 4.2. in  \cite{DV} gives the first claimed inequality.
\end{proof}

\begin{proof}[Proof of Theorem \ref{t:density}]
Without loss of generality, assume that $A=[0,b]$, and let $\ell$ be the Lipschitz constant of $f$. We first set up left initial conditions for Brownian last passage across $n$ lines so that the last passage values to the top line approximate $f(-\cdot)$.

The starting location will be $-a=-a_n$, so the typical curve for last passage across driftless Brownian motions of variance 2 is $\sqrt{8n(x+a)}$. The drift we use is the negative of the slope $\sqrt{2n/a}$ at $0$, which creates a curve close to the downwards parabola $\sqrt{8na}-x^2\sqrt{n/(2a)^3}$. Heuristically, if $n/a^3\to \infty$,  this solves the first step of the problem for $f=\zero_0$.

Let $\Lambda_{n,m}(x)$ be the last passage through the top $n-m$ lines. This has typical value $$\lambda_{n,m}(x)=\sqrt{8(n-m)(x+a)}-x\sqrt{2n/a}.$$
The curve $\lambda_{n,m}$ takes its maximum $\sqrt{8an}-m\sqrt{2a/n}$ at $-am/n$. When $m$ is small, it also behaves like a downwards parabola $-x^2\sqrt{n/(2a)^3}$ nearby. So we can use $\Lambda_{n,m}$ as a narrow at this location.

Note that $\Lambda_{n,m}$ has the same law as a shifted and scaled, but not tilted version of $\Lambda'$ considered in Proposition \ref{p:wedge}. The scaling is Brownian, except the factor-of-two speedup from standard Brownian motion.  For the argument any sequence $a_n=n^\gamma$ with $\gamma\in (0,1/3)$ would work; we choose $a=a_n=n^{1/4}$. Let $m_n=\lfloor bn/a\rfloor$.
A Borel-Cantelli argument with Proposition \ref{p:wedge} gives that for all large $n$ and all $m=0,\ldots, m_n$, and with the shorthand $y_m=-am/n$ we have
\begin{equation}\label{e:Lambdanm}
-d_n+\zero_{[\pm\eps_n]}(x-y_m)\le  \Lambda_{n,m}(x) -\sqrt{8an}+m\sqrt{2a/n}\le d_n+c_n|x-y_m|
\end{equation}
where $c_n\to\infty, d_n\to 0$ are deterministic sequences, and $\eps_n=a/n$.
For integers $m=0,\ldots ,m_n$ we set the left initial conditions  $$f_{n,m}=f(am/n)-\sqrt{8an}+m\sqrt{2a/n}+d_n+\eps_n/\ell$$

If $n/a\ge 2 \ell^2$ then for all $m\in[1,bn/a]$ we have $f_{n,m+1}\ge f_{n,m}+\sqrt{a/(2n)}$, in particular $f_{n,m}$ is increasing in $m$. We write the Brownian last passage value as
$$
H_n(x)=\max_{m=1}^{m_n}\,f_{n,m}+\Lambda_{n,m}(x)
$$
and by \eqref{e:Lambdanm} we have
$$
\max_{m=1}^{m_n}\zero_{[\pm\eps_n]}(x-y_m) +f(-y_m)+\eps_n/\ell \le H_n(x) \le 2d_n+\max_{i={1,\ldots m_n}\atop y=ma/n}f(-y_m)-c_n|y_m-x|.
$$
When $c_n\ge \ell$ this implies the following with
$$
f(-x)+\zero_{[0,b]}(-x)\le H_n(x)\le \bar f_n(-x)$$
where
$$
\bar f_n(x)=2d_n+\eps_n/\ell+
 \begin{cases}
  f(x), &x\in (-b,0),\\
  -c_n |x|+\max  f, & x\le0, \\
  -c_n |x-b| + \max  f, & x\ge b .
 \end{cases}
$$
The upper bound implies that for every $\beta<c_n$, we have
$
H_n(x)\le -\beta|x|+\beta b
+{\max} f
$
so $H_n(-\cdot)$ satisfies the claims (i)-(ii) of the theorem.

Next, we modify the initial conditions $f_{n,m}$ slightly; namely for $\delta$ small we set
$$
F_{n,m}=\begin{cases}
f_{n,m}+R_m, & m\le m_n,\\
f_{n,m_n} +\eps m + R_m, & m\ge m_n.
\end{cases}
$$
where $R_m$ are independent and are $n-m$-fold convolutions of Uniform $[0,\delta]$ random variables. We have
$$
H_n(x)\le G_n(x):=\max_{m=1}^n F_{n,m}+\lambda_{n,m}(x) \le H_n(x) + \delta n + \max_{m=1}^{n} R_n \le H_n(x) + 2\delta n
$$
where we used that $\lambda_{n,m}$ is nonincreasing in $m$. We choose $\delta=\delta_n$ small enough so that $\delta_n n \to 0$ with $n$ and that $F_{n,m}$ is almost surely increasing in $m$. By Theorem \ref{t:mix}, the law of $G_n$ is a finite linear combination of parallel edge processes.

The classical time inversion takes a function $g:(0,\infty) \to \mathbb R$ to $t g(1/t)$. It preserves the law of Brownian motion, see Theorem 1.9 \cite{morters2010brownian}. It takes the linear function $\nu t+\mu$ to linear function $\mu t+\nu$. Since time inversion is a linear operator, it follows that it takes Brownian motion started at $\mu$ with drift $\nu$ to Brownian motion started at $\nu$ with drift $\mu$.

Brownian motions with increasing initial condition  vector $\mu$ and increasing drift vector $\nu$ have a positive probability of never intersecting. Applying time inversion to this ensemble we see that $\mu$ and $\nu$ get switched. Taking limits, we see that the same holds even when $\mu,\nu$ are merely nondecreasing.

For functions on $(-a,\infty)$ the time inversion operator
$$
I_ag (t)=\frac{t+a}{a}g\left(-\frac{at}{t+a}\right), \qquad t\in (-a,\infty).
$$
is just a conjugate of the classical time inversion by shift and Brownian scaling. $I_a$ transforms graphs of functions by reparametrizing the underlying set $\mathbb R^2$. For fixed compact sets this reparametrization converges to $(x,y)\mapsto (-x,y)$ uniformly as $a\to\infty$. Using this it is easy to check that $I_af(-\cdot)\to f$ and also $I_aG_n\to f$ in $\uc$ and compactly on $A$. This shows (i).

$I_a$ maps the set of linear functions to itself. It fixes the zero intercept $d$ and maps the slope $\nu\mapsto d/a-\nu$. Since $a\to\infty$, and $G_n\le d-\beta|x|$, we see that $F_n\le d-\beta|x|/2$ for all large $n$, showing (ii).

$I_a$ maps parallel edge processes to pointed ones, with $(h,\nu)\mapsto (a\nu,h/a)$ showing (iii).
\end{proof}

Let $\uc_0$ be the space of $\uc$ that are $-\infty$ outside some compact set. The following lemma will be used to remove the Lipschitz and $\uc_0$ conditions needed in Theorem \ref{t:density}.

\begin{lemma} \label{l:usc0} Let $\cS_1,\cS_2: \mathbb R^2 \mapsto \mathbb R\cup \{-\infty\}$ be random upper semicontinuous functions.

If $\cS_1,\cS_2$ have the same distance laws on  $\uc_0^2$ then they also
have the same distance laws on $\uc^2$. This also holds if $\uc_0$ is replaced by functions that are supported on an interval and are Lipschitz there.
\end{lemma}
\begin{proof}
For a set $A$ defined the restriction $f_{|A}$ of $f\in \uc$ as $f$ on $A$ and $-\infty$ outside.

If $f_n\uparrow f, g_n\uparrow g$ pointwise, then a.s.
$$
f_n\cdot \cS \cdot g_n=\sup f_n+\cS+g_n \;\;\rightarrow \;\; \sup f+\cS+g =f\cdot S\cdot g.
$$
This implies convergence in distribution. Approximating $f,g$ from below by restrictions, we see that distributions of $f\cdot \cS \cdot g$ on $\uc_0^2$ determine those  on $\uc^2$. For $f$ whose support is contained in a compact interval  $A$, approximate from above by functions supported on $A$ that are Lipschitz there.
\end{proof}

We are now ready to prove Theorem \ref{t:char}
It can be used to characterize the KPZ fixed point. In those cases $\cS_2$ is the Airy sheet.

\subsection*{Proof of the characterization theorem }

Let $f,g\in \uc$ be defined on compact intervals $A,B$ and Lipschitz there. By Lemma \ref{l:usc0}, it suffices to show that the information given determines the law $f\cdot S\cdot g$. We use the shorthand $f+\cS+g$ for the function $f(x)+\cS(x,y)+g(y)$.
Let $f_n\to f, g_n\to g$ in $\uc$ and compactly in on $A,B$ wtih $f_n \ge f$ and $g_n \ge g$, and so that for every $b>0$ and sufficiently large  $n$ the quantities  $\sup f_n+b|x|$, $\sup g_n+b|x|$ are bounded in $n$.
Then $f_n+\cS+g_n \to f+\cS+g$ from above in $\uc(\R^2)$. Note that addition is not continuous in $\uc$, but the compact convergence ensures that this holds nevertheless.

Let $C$ be random so that $\cS(u)\le C(\|u\|_1+1)$, and let $D,N$ be random, depending on $C$ so that $f_n+2C|x|$, $g_n+2C|x|\le D$ for all $x$, and $n>N$.
We have, by the above
$$f_n+\cS+g_n\le D+C-C(|x|+|y|)
$$
this and convergence in $\uc(\R^2)$ implies that
$\sup f_n+\cS+g_n \to \sup f+\cS+g
$; this is easy to check using the fact that the location of the supremum has to stay  bounded.
This also works when $F_n, G_n(\cdot -)$ are the random functions given by Theorem \ref{t:density} approximating $f$ and $g(\cdot -)$. So
$$
F_n\cdot \cS\cdot G_n \to  f\cdot \cS \cdot g
$$
almost surely, so also in distribution.  \qed

\section{A general convergence theorem}\label{s:first}

In continuous polymer models, the functional $f\mapsto \sup f$ of last passage is replaced by
\begin{equation}\label{e:ncirc}
\supn f=\tfrac{1}{n}\log \int e^{n f(x)} \,dx, \qquad f\ncirc g =\, \supn (f+g).
\end{equation}
Metric composition is as implied: $\cK \ncirc \cS\,\,(x,y) =  \cK(x,\cdot) \ncirc \cS(\cdot,y)$.

Unlike lass passage percolation, polymer models may ignore isolated large values at metric composition. We need a concept to deal with this issue. We say that $f$ is $\eps${\bf-thick} at $x$ if
$$
\underset{y:|y-x|\le \eps}{\supn}f(y)\ge f(x)-\eps$$ The {\bf thickness constant} $\theta_n(f|A)$ of $f$ on $A$ is the infimum of $\eps$ so that $f$ is $\eps$-thick at all $x\in A$.

\begin{theorem}\label{t:first}
Assume that we have a sequence of random upper semicontinuous functions $\cK_n:\mathbb R^2\to \mathbb R\cup\{-\infty\}$ so that the following holds.
\begin{enumerate}[(i)]
    \item Thickness: for every $A$ compact $\theta_n(\cK_n|A)\cp 0$.
    \item Uniform growth: for some $b,b'$ reals, $(b|\cdot |)\ncirc\; \cK_n \ncirc (b'|\cdot|)$ is tight.
    \item Test function convergence:  for every $a,a'\in \R$, every sufficiently large negative $h\in \mathbb R$, all vectors $\nu,\nu'$ with a sufficiently small maximum there exists independent $F_n,G_n$, converging compactly in law, to $F\sim \ep(a,h,\nu), G(\cdot -)\sim\ep(a',h',\nu')$ so that
    \begin{equation}\label{e:bbp-conv}F_n\ncirc \cK_n \ncirc G_n\cd F\cdot \cS \cdot  G,
    \end{equation}
    and $\sup_x F_n+2b|x|$, $\sup_x G_n+2b|x|$ are tight form above.
\end{enumerate}
Then the $\cK_n$ are tight and the distance laws of every limit point of $\cK_n$ are given by those of $\cS$. If for some $b'>b$ and every compact interval $A$ the quantity $\sup_A (b'|\cdot|) \ncirc \cK_n$ is tight, then for every continuous initial condition $h$ with $\sup h(x)-b|x|<\infty$ we have $h\,\ncirc\, \cK_n \cd h\,\cdot \,\cS$ in $\uc$.
\end{theorem}

\begin{proof}
We first show that when $\cK_n$ is a deterministic sequence satisfying the deterministic version of (i) and the dereministic version of (ii) given by
$$
\supn \cK_n+m \le d, \qquad m(x,y)=b|x|+b'|y|,
$$
then $\cK_n$ is precompact in $\uc(\R^2)$. To see this, let $A_0$ be compact, $\theta>0$ and $A$ the $\theta$-fattening of $A$. Let $u \in A_0$.
When $\theta_n(\cK_n,A)<\theta$, by the monotonicity of $\supn$ we have
\begin{equation}\label{e:thickness}
\cK_n(u) \le \theta+ \underset{A}{\supn }\;\cK_n\le \theta+ d-\inf_{A}g
\end{equation}
a bound that is uniform on $A_0$.

Returning to the original setting, by Prokohorov's  theorem we have tightness of $\cK_n$ in $\uc(\R^2)$, and we consider a distributional limit point $\cK_n \to \cK$ along a subsequence $N$. By the Skorokhod representation, we set up convergence so that this and the convergences in (ii) for fixed parameters hold almost surely, and the tightness conditions are replaced by a random constant bound.  In particular, we can have $\theta_n(\cK_n|A)\to 0$ a.s.\ for all closed balls $A$ of integer radius around $0$.

We will now work along the subsequence $N$. Below, the $F$s will be functions of the first variable and $G$s of the second: $F+G$ means the function $(x,y)\mapsto F(x)+G(y)$. Since $\cK_n $ converges in $\uc(\R^2)$ and $F_n+G_n$ converges compactly on $\R^2$, we have that
$$
U_n:=F_n+\cK_n+G_n \to F+\cK+G=:U
$$
in $\uc(\R^2)$. We consider the box $A=[\pm r]^2$. We have
\begin{align}\nonumber
\underset{A^c}{\supn}\; U_n& =\underset{A^c}{\supn}\; (F_n+G_n-m)+(U_n+m)\le \underset{A^c}{\supn}\;  (\sup_{A^c} F_n+G_n-m)+(U_n+m) \\&\le -cr + \underset{A^c}{\supn}\; {U_n+m}\le -cr + D\label{e:supnu}
\end{align}
Consider strictly increasing boxes $A,A',A''$ of the form $[\pm r]^2$.
By the thickness of $\cK_n$ and compact convergence of $F,G$ to conitnuous limits, $\theta(U_n|A')\to 0$. This with the argument for \eqref{e:thickness} implies the first inequality of
$$\liminf \underset{A''}{\supn}\; U_n \ge \liminf \sup_{A'} \,U_n \ge \sup_A U,$$ and the second follows by $\uc$ convergence. On the other hand, since $A$ is compact,
$$
\limsup \underset{A}{\supn}\; U_n \le \limsup \underset{A}{\sup}\; U_n \le \underset{A}{\sup}\; U.
$$
Together with \eqref{e:supnu} these inqualities imply $\supn U_n \to \sup U$.
%
%
%
%The sequence $U_n(u)+c\|u\|$ is bounded, so $\sup U_n \to \sup U$, and the location $z_n\in\R^2$ of the maximum stays in some closed ball $A$ around 0.  Let $\eps>0$, let $z$ be a limit point of $z_n$, and let $B$ be a neighborhood of $z$ so that $\inf_B F_n+G_n\ge (F_n+G_n)(z_n)-\eps$ for all large enough $n$; this exists because of compact convergence.
%$$
%\supn \,U_n \ge \underset{B}{\supn}\; U_n \ge (\inf_{B} F+G)+\underset{B}{\supn} \,\cK_n \ge (F+G)(z_n)-\eps+\underset{B}{\supn}\,\cK_n
%$$
%the last inequality holds for large $n$ by construction of $B$. Since $\theta(\cK_n,A)\to 0$, the last $\supn$ term can be bounded below by $\cK_n(z_n)-\eps$ for large enough $n$, giving $\supn U_n\ge \sup U_n-2\eps$. Taking limits along the subsequence where $z_n\to z$ gives $\limsup \supn \, U_n \ge \sup\, U$.
%For the other direction, note $\supn U_n\le ({\supn}_{A^c} U_n)^+ +{\supn}_A U_n$ for all compact balls $A$. The first quantity can be made uniform by taking $A$ large by sublinearity of $U_n$. The second is bounded by $\sup_A U_n + o(1)$. This gives $\supn\, U_n \to \sup U$.

By (iii) this implies that $F\cdot \cK\cdot G\ed F\cdot\cS \cdot G$ for the test functions in (iii).
Now the Characterization theorem, Theorem \ref{t:char} implies that $\cK$ and $\cS$ have the same distance laws. Note that in the approximations used in the Characerization theroem, $\max_\nu$ can be arbitrarily large and negative.

The argument above showed $F_n\ncirc \cK_n \ncirc G_n \to F\cdot \cK \cdot G$. Turning to the last claim, essentially the same argument using the given bounds also gives $h\ncirc \cK_n \cdot g\to h \cdot \cK \cdot g$ for any $g$ compactly defined and continuous on $g^{-1}(\R)$.
This implies that  $h\ncirc \cK_n \to h\cdot \cK$ in $\uc$. But $h\cdot \cK \ed h\cdot \cS$ since $\cK,\cS$ have the same distance laws.
\end{proof}

\section{A convergence theorem based on the Burke property}

In this section, we prove a variant of Theorem \ref{t:first} which is more tailored to the polymer models we consider. It reduces tightness requirements to the minimum. While for the KPZ equation, the tightness needed for strong convergence has been established previously, this could be a daunting project to carry out in other cases.

There are {\it only} two extra inputs needed, beyond convergence of BBP marginals. The first is the {\bf metric Burke property} satisfied by all models we consider. This can be stated in terms of {\bf line metrics} of functions, see \eqref{e:line metric}.

We proceed by example. The O'Connell-Yor model is simply a composition of line metrics of Brownian motion $\cO_n=\cB_1\circ_1 \cdots \circ_1\cB_n$ The metric Burke property in this case states that independent Brownian line metrics with different drifts commute in law: $\cB_\nu\circ_1 \cB_\mu\ed \cB_\mu \circ_1 \cB_\nu$ as functions of $x,y$. This can then be iterated, and gives $\cB_\nu\circ_1 \cO_n\ed \cO_n \circ_1 \cB_\nu$. A consequence is the more usual Burke property: $B_\nu \circ_1 \cB_\mu\ed B_\nu+C$, where $B_\nu$ is a two-sided Brownian path, and $\mu<\nu$, and $C$ is a random constant.

We note that taking limits of the metric Burke property in Brownian last passage shows that  $\cB_\nu \circ \cS \ed \cS \circ \cB_\nu$. Similarly, using tightness and convergence to the Airy sheet, one gets that such compositions with $\cS$ have BBP statistics.

In addition, we rely on tightness of the systems started from stationary initial conditions, in other words, tightness of the Baik-Rains statistics. In many cases, this follows from BBP marginals. It has been first established in much greater generality by Sepp\"al\"ainen's workshop see \cite{balazs2010order}, \cite{balazs2011fluctuation}, \cite{seppalainen2010bounds}, \cite{flores2014fluctuation}.

We will use the following simple fact. Let $f_n\to f$ in $\uc$, and assume that for every compact $A$
\begin{equation}\label{e:unifrom}
\lim_{\eps \downarrow 0}\; \limsup_{n\to \infty} \sup_{x,y\in A^2:x\le y \le x+\eps} f_n(x)-f_n(y)=0
\end{equation}
and $f$ is continuous. Then $f_n\to f$ compactly.

\subsection*{The proof of Theorem \ref{t:second}}

We first bound the growth of $\cK_n$ uniformly. Let $\nu\in \R$. By near-linearity, monotonicty and tightness at stationarity, now using $x$ instead of $\cdot$ as a dummy variable for functions arguments, we have
\begin{equation}\label{e:growth0}
(C-\eps |x |+\nu x)\ncirc \cK_n\;\le \;B_{\nu,n}\ncirc \cK_n \;\ed\; B'_{\nu,n}+C' \;\le \;C''+\eps |x|+\nu x
\end{equation}
with tight random constants.  Also, since $(f\vee g)\ncirc h \le (f \ncirc h) \vee ( g\ncirc h)+(\log 2)/n$, and \eqref{e:growth0} applied to $\pm (\nu+\eps)$ with $\nu\ge 0$ we get
\begin{equation}\label{e:growth}
(\nu |x|)\ncirc  \cK_n  \le C+ (\nu+2\eps) |x|.
\end{equation}
The $\nu=0$ case implies
\begin{equation}\label{e:sublin}
1 \ncirc\cK_n \ncirc -2|\eps x| \le C.
\end{equation}

Next, we verify the assumptions of Theorem \ref{t:first} for the ``smoothed'' kernel
$\cK_n'=\cB_{a,n} \ncirc \cK_n \ncirc \cB'_{a,n}$ for some $a$ large and negative. The metric Burke property (i) for $\cK_n$ implies (i) for $\cK_n'$.  This and (ii) for $\cK_n$ implies (ii) for $\cK_n'$, as long as $a<\nu$. So $\cK'_n$ satisfy the growth conditions above.

The following inequality is a consequence of two properties of $\supn$: monotonicity, and $\supn (f+a)=a+\supn f$ for constants $a$. For any line metric \eqref{e:line metric} $\cf$ with function $f$, any measurable function $g$,  any $h\ge 0$ and $D_h g(x)=g(x+h)-g(x)$ we have
\begin{equation}\label{e:c-bounds}
\cf \ncirc g\,(x-h)\ge \cf \ncirc g\,(x)+ D_h f(x-h), \quad g\ncirc \cf \,(x+h)\ge g\ncirc \cf\,(x)+D_h f(x).
\end{equation}
Applying this to $B_{a,n}$ and $\cK'$ twice, we get that on any set
$$
\cK'(x-h,y+b)\ge \cK'(x,y)+D_hB_{a,n}(x-h)+D_yB_{a,n}'(y)
$$
and for any small rectangle $A=u+[\pm \eps]^2$
$$
\supn_A \cK'_n \ge \cK'_n(u) -m(\eps)-m'(\eps)+2(\log 2\eps)/n
$$
where $m,m'$ are the uniform moduli of continuity of $B_{a,n},B_{a,n}'$ in the $\eps$-fattening of $A$. Compact convergence of $B,B'$ in law now implies that $\theta(\cK'_n|A)\cp 0$, verifying (i).

The metric Burke property and the BBP asymptotics assumption verifies (iii). This step uses the simple fact that $\cB_{\nu_1}\ncirc \cdots \ncirc \cB_{\nu_k}\to \cB_{\nu_1} \cdots \cB_{\nu_k}$ compactly in law with the required growth conditions.

Theorem \ref{t:first} with \eqref{e:growth} shows that that for every continuous $h$ with $\sup h(x)/(1+|x|)<\infty$, we have that  +
$F_n=h\ncirc \cK'_n$ converges in law in $\uc$ to the continuous $F=h\cdot \cB_a\cdot \cS\cdot \cB'_a$. We now upgrade to compact convergence. For this note, that $F_n=G_n \ncirc \cB'_{a,n}$ with $G_n=h\ncirc \cB_n \ncirc \cK_n$. Since $F_n$ is tight in $\uc$ and $B'_{a,n}$ is tight for compact convergence, there exists joint subsequential limits.

We use Skorokhod representation to get almost sure convergence. The second bound in \eqref{e:c-bounds} applied to $F,B_{a,n}$ gives
$$
F_n(x+h) \ge  F_n(x) + D_hB_{a,n}(x)
$$
now compact convergence of the $B_{a,n}$ and the fact \eqref{e:unifrom} implies that $F_n\to F$ compactly a.s. along the subsequence. It follows that as $n\to\infty$
\begin{equation}\label{e:F}
F_n\cd F \qquad \mbox{compactly}.
\end{equation}
Write $H_{a,n}=h\ncirc \cB_{a,n}\ncirc \cB_{a,n}''$, $H_a=h\cdot \cB_a \cdot \cB''_a$ and use the Burke property to restate \eqref{e:F} as
$$
H_{a,n} \ncirc \cK_n \cd H_a \cdot \cS \qquad \mbox{compactly.}
$$
The continuity and the tails bounds of $h$ implies that for large enough $a$, compactly,
\begin{equation}\label{e:Hconv}
H_{a,n}\cd H_a, \quad \mbox{ as }n\to \infty \qquad  \mbox{ and } \qquad H_a \cd h, \quad \mbox{ as } a\to  -\infty.
\end{equation}
Consider an interval  $A=[-r,r]$.  With supn-s taken in the  first variable of $\cK_n$ we have
\begin{equation}\label{e:onA}
 |(\underset{A}{\supn}\; H_{n}+\cK_n)-(\underset{A}{\supn}\;h + \cK_n) | \le \sup_A |H_{a,n}-h| \cd \sup_A |H_a-h|
\end{equation}
as $n\to\infty$ by \eqref{e:Hconv}. In turn, $\sup_A |H_a-h|\cp 0$ as $a\to 0$ by \eqref{e:Hconv}.

The bound \eqref{e:growth} implies that for any compact set $J$ we have
\begin{equation}\label{e:offA}
\sup_{y\in J}\;\underset{x\in A^c}{\supn}\; h(x) + \cK_n(x,y) , \quad \sup_{y\in J}\;\underset{x\in A^c}{\supn}\; H_{a,n} + \cK_n(x,y) \le C-C'r, \qquad \forall r\ge 0,
\end{equation}
where $C,C'$ are tight in $n$. Since $B_{a,n}$ has stochasticly increasing increments in $a$, and so the law of $H_{a,n}$ is monotone in $a$, the same constants work as $a\to -\infty$.

We are now ready to prove $h\ncirc \cK_n\cd h\cdot S$.  Given an open  neighborhood $U$ of the law $h\cdot \cS$ with respect to compact convergence, we pick $a_0$ so that $H_a\cdot S$ has law in $U$ for $a>a_0$.

It suffices to show that for all $\delta>0$ and compact interval  $J$, there is $a>a_0$ so that for all large enough $n$, $P(|H_{a,n}\ncirc \cK_{a,n}-h\cdot \cK_n|\one_J>c\delta)\le c\delta$ with $c$ universal. We say an event is likely if it has probability at least $1-c\delta$ for some universal $c$.

First pick $m$ so that likely $\sup  h\cdot S\ge -m$. Since $H_a\ge h$, the event $\sup H_a\cdot S\ge -m$ is likely as well.
Now pick $r$ large enough so that the quantities in \eqref{e:offA} are likely at most $-2m$. Then pick $a<a_0$ large negative enough so that we likely have $\sup_A|H_a-h|\le \delta $, then pick $n_0$ so that for $n\ge n_0$ we likely have $\supn_A |H_{a,n}-h|<2\delta$. When these events hold, we have
$|H_{a,n}\ncirc \cK_n-h\ncirc \cK_n|\one_J<2\delta$, as required. $\qed$

\begin{remark}\label{r:I}
Let $\mathcal I$ be the set of initial conditions for which the convergence conclusion of Theorem \ref{t:second} holds.  $\mathcal I$ contains all continuous functions that grow at most linearly. Note that $\mathcal I$ contains $h:\mathbb R\to \mathbb R\cup \{-\infty\}$, not identically $-\infty$,  with the same growth condition but with only $e^h$ continuous (we ``compactifty'' at $-\infty$). To see this,  approximate $h$ by $h_m=h\vee m$, for integers $m$ tending to $-\infty$. The theorem holds for $h_m$, and  note that $F_n=h\ncirc \cK_n \le h_m \ncirc \cK_n \le (h\ncirc \cK_m)\wedge (n\ncirc \cK_n)+(\log 2)/\log n$. Let $G_n=0\ncirc \cK_m$, $G=0\cdot \cS$, and  Let $F=h\cdot \cS$, which is real-valued. Thus, using Skorokohod representation, along subsequences,  for every negative $m$, $F_n\vee (G_n+m)\to F\vee (G+m)$ compactly and $G_n\to G$ compactly, which implies that $F_n \to F$ compactly.

Next, note also that for any function $h$ for which there are sequences $f_n\le h \le g_n\in \mathcal I$ so that $f_n\cdot \cS,g_n\cdot \cS\to h\cdot \cS$ is also in $\mathcal I$, by a monotonicity argument. This covers all classicly considered initial conditions, such as half-Brownian or $\zero_{[0,\infty)}$. What is missing are narrow wedges, for which the claim of the theorem has to be modified. \end{remark}

Using the results and ideas from \cite{LIS} we can extend Theorem \ref{t:second} to subsequential compact convergence of the $\cK_n$. This implies convergence for all classical initial conditions, including combinations of narrow wedges. For this, recall the quadrangle inequality satisfied by the Airy sheet, see \cite{DOV}
$$
\cS(x,y)+\cS(x',y') \ge \cS(x,y')+ \cS(x',y), \qquad \mbox{ for } x\le x', y\le y'.
$$
For discrete deterministic polymers, the Lindstr\"om–Gessel–Viennot formula for the number of nonintersecting paths $x\to x', y\to y'$ is a determinant of a $2\times 2$ matrix. This is nonnegative, which gives the quadrangle inequality. For continuum models, this can be obtained as a limit.

\begin{remark}\label{r:LIS} In \cite{LIS}, sequences of random $\cK_n\in \uc(\R^2)$ satisfying the triangle inequality are considered. Assume that for ever compact $A$, the function $\cK_n$ is continuous on $A$ for large $n$.  If for every rational $x$, the sequences $\cK_n(x,\cdot)$ and $\cK_n(\cdot, x)$ are tight with respect to compact convergence, then $\cK_n$ is tight with respect to compact convergence. The proof is straightforward, and we do not reproduce it here.
\end{remark}

\begin{corollary}[Compact convergence]\label{c:compact}
Assume that in addition to the conditions of Theorem \ref{t:second}, $\cK_n$ also satisfies the quadrangle inequality for each $n$.
Then $\cK_n$ is tight with respect to compact convergence, and each limit point is continuous and has the same distance laws as $\cS$.
\end{corollary}
\begin{proof}
The definition of $\supn$ and the quadrangle inequality implies that for any initial condition $h_+$ supported above $x$,
$$
\cK_n(x,y) + h_+\ncirc \cK_n(y') \ge
\cK_n(x,y') + h_+\ncirc \cK_n(y),
$$
and, for $h_-$ supported below $x$,
$$
h_-\ncirc \cK_n(y) +\cK_n(x,y') \ge
h_-\ncirc \cK_n(y')+\cK_n(x,y).
$$
Now fix $x,y'$ and and
set $W_n=\cK_n(x,y')$, $F_{n\pm}(y)=h_{\pm}\ncirc \cK_n(y)-h_{\pm}\ncirc \cK_n(y')$.
Then for $y\le y'$
$$
F_{n-}(y)+G_n \le \cK_n(x,y) \le F_{n+}(y)+G_n
$$
Now by Remark \ref{r:I} we can pick $h_\pm$ to be close to $x$ so that  $F_{n\pm}$ converge compactly, and their
limit $h_\pm \cdot S$ are close to $\cS(x,\cdot)$. It follows that $\cK(x,\cdot)$ converges compactly in law to $\cS(x,\cdot)$, at least for $y\le y'$, but $y'$ was arbitrary. Thus for every  $x$, the sequence $\cK_n(x,\cdot)$ is tight and by a symmetric argument, so is $\cK_n(\cdot,x)$. The claim now follows from Remark \ref{r:LIS}.
\end{proof}

\section{Application to two models}
\label{s:application}

Theorem \ref{t:char} is more general than needed for our applications. In particular, conditions (i), (iii), (iv) are well known or straightforward, and so is the quadrangle inequality. We show the single time version of Theorem \ref{t:main}. The multiple time result is a consequence of the Markov structure in these models.

\subsection*{The KPZ equation}
The Cole-Hopf transformation turns the KPZ equation into the stochastic heat equation,
$$
\dot Z=\tfrac12 Z''+ZW,
$$
where $W$ is standard space-time white noise, see the notes  of \cite{quastel2011introduction} for details.
Let $Z(x;y,t)$ be the value at time $t$ and location $y$ of the solution of the stochastic heat equation with initial condition $\delta_x$ at time $0$.
Let us take logarithms to express the {\it free energy}, center and 1-2-3 scale to get
$$
\cK_n(x,y)=\frac{\log Z(2n^2x;2n^2y, 2n^3)+n^3/12}{n},
$$
and let
$$
\supn f = \frac{1}{n}\log \int e^{nf},
$$
and recall $\ncirc$, \eqref{e:ncirc}. Let $\cB_\nu$ be the Brownian metric: $\cB_\nu(x,y)=B_\nu(y)-B_\nu(x)$ when $x\le y$ and $-\infty$ otherwise, where $B_\nu$ is Brownian motion with drift $\nu$ and variance 2. Repeated $\cB_\nu$s will be assumed independent.

%For a vector $\nu$, define $\cB_{\nu,n}=\cB_{\nu_1}\ncirc \cdots \ncirc\cB_{\nu_k}$, and $\cB_\nu=\cB_{\nu_1}\cdot \cB_{\nu_2} \cdots  \cB_{\nu_k}$. Fixing the Brownian motions and letting $n\to \infty$, it is easy to check that for $\nu$ fixed, almost surely, $\cB_{\nu,n}(x,\cdot)\to \cB_\nu(x,\cdot)$ compactly on $[x,\infty)$, and if $b>\max \nu$ then $\sup_{y\ge x} \cB_{\nu,n}(x,y)-by$ is bounded in $n$.

Translated to our language, Corollary 1.15 of \cite{borodin2014free} shows that for any $\nu,x,y$ fixed,
\begin{equation}\label{e:BBP-limit}
\cB_{\nu,n} \ncirc \cK_n (x,y) \cd  \cB_\nu \cdot \cS(x,y).
\end{equation}
The tightness of the shift of the Brownian stationary measures was shown in \cite{balazs2011fluctuation}. The stated convergence follows from Theorem \ref{t:second}. Note that limit can be taken along any $n\to\infty$, not just integers.

\subsection*{The O'Connell-Yor polymer}

See \cite{o2012directed} and \cite{o2001brownian} for background.
The O'Connell-Yor polymer, and  free energy can be defined using the above notation by
$$
F(x;y,k)=\cB^1 \circ_1  \cdots \circ_1 \cB^k(x,y),$$ where the $\cB^i$ are independent with drift $0$, and, for this defintion only, varianca 1. Following \cite{borodin2014free}, consider $\Psi(\theta)=(\log\Gamma(\theta))'$. The derivative $\Psi'$ is a continuous bijection $\R^+\to\mathbb \R^+$, and we will take limits in the asymptotic (space,time) direction $(1,\Psi'(\theta))$.
We scale values by $n$ to to be consistent with $\supn$. We set
$$
\cK_n(x,y)=\frac{F(2xn^2;bn^3 + 2yn^2,an^3)-cn^3+2(x-y)n^2\Psi'}{n}, \quad a=-2/\Psi'', \; b=a\Psi', \;c=\theta b - a\Psi.
$$
Again, as $an^3\to \infty$ along integer values, Theorem 1.3 (b) of \cite{borodin2014free} amounts to \eqref{e:BBP-limit} for this case as well. The distributions on both side depend only on $x-y$, so we may assume $x=0$. To be precise, Theorem 1.3 shows these claims for $y=0$, as it addresses asymptotics precisely on straight lines of the spacetime parameters.
However, this is merely for convenience -- their arguments give \eqref{e:BBP-limit} for the very slightly off-center $y$ as well. The verification now proceeds exactly as for the KPZ equation, with tightness, in this case, shown in \cite{seppalainen2010bounds}.

\smallskip

\noindent {\bf Acknowledgments.}
B.V. was supported by the Canada Research Chair program, the NSERC Discovery Accelerator grant, and the MTA Momentum Random Spectra research group.

\bibliographystyle{dcu}
\bibliography{heat}

\end{document}